\documentclass{article}
\pdfoutput=1

\usepackage{amsmath}
\usepackage{amsthm}

\usepackage{enumerate}

\usepackage{hyperref}

\hypersetup{pdftitle={Asymptotically sharp Markov and Schur inequalities on general sets},
 pdfauthor={Sergei Kalmykov, Béla Nagy and Vilmos Totik}}

\newcommand*{\bR}{\mathbf{R}}
\newcommand*{\Bal}{\mathrm{Bal}}

\newtheorem{Th}{Theorem}[section]

\newtheorem{Prop}[Th]{Proposition}
\newtheorem{Lem}[Th]{Lemma}

\def\sep{\;\vrule\;}

\def\c{{\rm cap}}

\begin{document}

\title{Asymptotically sharp Markov and Schur inequalities on general sets}

\author{Sergei Kalmykov
 \\
Bolyai Institute, University of Szeged; \\
Far Eastern Federal University; and\\
Institute of Applied Mathematics, FEBRAS
\and
B\'ela Nagy
\\
MTA-SZTE Analysis and Stochastics Research Group, 
\\Bolyai Institute, University of Szeged
\and
Vilmos Totik
\\
Bolyai Institute, University of Szeged;
and\\
Department of Mathematics and Statistics, \\
University of South Florida
}

\date{}

\maketitle

\begin{abstract}
Markov's inequality for algebraic polynomials on $\left[-1,1\right]$
goes back to more than a century and it is widely used in approximation theory.
Its asymptotically
sharp  form  for unions of finitely many intervals has been found
only in 2001 by the third author.
In this paper we extend this asymptotic form to arbitrary compact
subsets of the real line satisfying an interval condition.
With the same method a sharp local version of Schur's inequality
is given for such sets.
\end{abstract}

\section{Results}

Markov's inequality is one of the most fundamental results in
approximation theory,  it states that  if $P_{n}$
is an algebraic polynomial of degree $n$,
then
\begin{equation}
\left\Vert P_{n}'\right\Vert_{[-1,1]} \le n^{2}\left\Vert P_{n}\right\Vert_{[-1,1]}, \label{eq:classicmarkov}
\end{equation}
where $\left\Vert \cdot\right\Vert_{[-1,1]} $ is the sup-norm over $\left[-1,1\right]$.
It is sharp, for the classical Chebyshev polynomials equality holds.

When one considers the analogue of (\ref{eq:classicmarkov}) on a set $K$ consisting of several
intervals, a new feature emerges: if $a_j$, $j=1,\ldots,2m$ are the endpoints
of the intervals that make up $K$, then around each $a_j$ there is a local Markov inequality
\[
\left\Vert P_{n}'\right\Vert_{[a_j-\varepsilon,a_j+\varepsilon]\cap K} 
\le (1+o(1))M(K,a_j)n^{2}\left\Vert P_{n}\right\Vert_{K},
\]
with some best constants $M(K,a_j)$, where $o(1)$ tends to 0 uniformly in $P_n$ as the degree $n$
tends to infinity. In general, these local Markov factors $M(K,a_j)$
are different, they depend on the location of $a_j$ in the set $K$.
The paper \cite{Totikacta} gave an explicit expression for them.
The asymptotically sharp global Markov inequality
\[\left\Vert P_{n}'\right\Vert_{K} \le (1+o(1))
\Bigl (\max_j M(K,a_j)\Bigr)n^{2}\left\Vert P_{n}\right\Vert_{K}\]
is then an immediate consequence.

The aim of this paper is to prove a sharp local version of (\ref{eq:classicmarkov})
for general compact subsets of $\mathbf{R}$ rather than for $[-1,1]$ or for sets consisting of
finitely many intervals. To this end we call the point $a$
a right-endpoint of the compact set $K\subset \mathbf{R}$ if there is a $\rho>0$ such that
\begin{equation}
\left[a-2\rho,a\right]\subset K\mbox{\ \  and \ \ }(a,a+2\rho)\cap K=\emptyset.\label{intc}
\end{equation}
We shall refer to (\ref{intc}) as the interval condition.
The numbers $a,\rho$ will be fixed for the whole paper, and we shall always
assume that $K\subset\mathbf{R}$ satisfies this condition.

We introduce the (asymptotic) Markov factor for $K$ at the endpoint $a$ as
\begin{equation}
M\left(K,a\right):=\limsup_{n\rightarrow\infty}\ \sup\left\{ \frac{\|P_{n}'\|_{[a-\rho,a]}}{n^{2}\left\Vert P_{n}\right\Vert _{K}}\sep \ \deg\left(P_{n}\right)\le n\right\}. \label{def:markovfact}
\end{equation}
Without changing the value of $M(K,a)$,
the norm in the numerator could have been taken
instead of $[a-\rho,a]$ on any interval $[a-\eta,a]$ so long as $[a-\eta-\varepsilon,a]\subset K$
for some $\varepsilon>0$. This is because on compact
subsets of the interior ${\rm Int}(K)$ of $K$ the norm of $P_n'$ is at most
$Cn\|P_n\|_K$ with some
$C$ (depending on the compact subset), see (\ref{Bernstein}) below.
In a similar manner, $M(K,a)$ would not change if
we used $|P_n'(a)|$ in  (\ref{def:markovfact}) instead of
$\|P_{n}'\|_{[a-\rho,a]}$. This is not absolutely trivial, but
it will follow from the considerations below.

To formulate the results, we need some potential theory and we
refer to the books \cite{Gardiner}, \cite{Landkof}, \cite{Ransford} or
 \cite{SaffTotik} for an introduction.
In particular, we denote the equilibrium measure
of $K\subset \mathbf{R}$ of positive logarithmic capacity $\c(K)>0$ by $\nu_{K}$.
This is absolutely continuous on the (one dimensional) interior of $K$,
and we
denote its density with respect to the Lebesgue-measure
by $\omega_{K}$: $\frac{d\nu_{K}\left(t\right)}{dt}=\omega_{K}\left(t\right)$.
If $K$ satisfies the interval condition (\ref{intc}) then, around $a$,
the density $\omega_K$ behaves like $1/\sqrt{|t-a|}$.
The Markov factor $M(K,a)$ is related to the quantity
\[
\Omega\left(K,a\right):=\lim_{t\rightarrow a-0}\omega_{K}\left(t\right)\left|t-a\right|^{1/2}.
\]
It will be proven in the next section that this limit exists, positive and finite, and with it we can state
\begin{Th}
\label{thmain} If $K\subset \mathbf{R}$ satisfies the interval condition (\ref{intc})
at a point $a\in K$,
then
\begin{equation} M(K,a)=2\pi^2 \Omega(K,a)^2.\label{eq1}\end{equation} 
\end{Th}

There is another problem which can be solved via the quantity $\Omega(K,a)$,
namely Schur's inequality on general sets. The original Schur inequality
(see e.g. \cite[Theorem 6.1.2.8]{Milovanovic})
claims that if $P_n$ is a polynomial of degree at most $n$ for which
\begin{equation} |P_n(x)|\le \frac{1}{\sqrt{1-x^2}},\qquad x\in (-1,1),\label{Schur0}\end{equation} 
then
\begin{equation} \|P_n\|_{[-1,1]}\le n+1.\label{Schur}\end{equation} 
The next theorem gives an asymptotically optimal
local version of this for general subsets of $\mathbf{R}$ rather than $[-1,1]$.
\begin{Th}
\label{thschur} 
Let $K$ be a compact subset of $\mathbf{R}$ with
the interval condition (\ref{intc}) at a point $a\in K$. 
Suppose that for polynomials
$P_n$ of degree $n=1,2,\ldots$ we have
\begin{equation} |P_n(x)|\le \frac{h(x)}{\sqrt{a-x}},\qquad x\in [a-\rho,a),\label{local}\end{equation} 
with some positive and continuous function $h$ on $[a-\rho,a]$, and assume also the global
condition
\begin{equation} 
\limsup_{n\rightarrow\infty} \|P_n\|_K^{1/n}\le 1.\label{global}\end{equation} 
Then
\begin{equation} \|P_n\|_{[a-\rho,a]}\le n(1+o(1))2\pi h(a)\Omega(K,a).\label{conc}\end{equation} 
This estimate is sharp for any $h$, for there are polynomials $P_n\not\equiv 0$ satisfying
(\ref{local}) and (\ref{global}) for which
\begin{equation} |P_n(a)|\ge n(1-o(1))2\pi h(a)\Omega(K,a).\label{conc1}\end{equation} 
\end{Th}
The $o(1)$ in (\ref{conc}) depends only on the function $h$ and on the
speed of convergence in (\ref{global}).

The condition (\ref{global}) is a very weak one, but something like
that is needed, for the polynomials $P_n$ cannot be arbitrary outside $[a-\rho,a]$:
just set $K=[-2,1]$, $a=1$, and with the classical Chebyshev polynomials
${\mathcal{T}}_n(x) = \cos(n\arccos x)$ consider $P_n(x)={\mathcal{T}}_{n+1}'(x)/(n+1)$. In this
case (\ref{local}) is true with $\rho=1$, $h(x)=1/\sqrt{1+x}$
(apply Bernstein's inequality (\ref{B}) below), but (\ref{conc}) does not hold
because $P_n(1)=n+1$ and $\Omega(K,1)=1/\pi\sqrt3$.

Schur's inequality (\ref{Schur}) is one way to deduce Markov's inequality
(\ref{eq:classicmarkov}) from Bernstein's inequality
\begin{equation}  |P_n'(x)|\le \frac{n}{\sqrt{1-x^2}}\|P_n\|_{[-1,1]},\qquad x\in (-1,1).\label{B}\end{equation} 
The same happens with (\ref{conc}) and the estimate
$M(K,a)\le 2\pi^2\Omega(K,a)^2$ in Theorem \ref{thmain}.
In fact, if $K\subset \mathbf{R}$ is a regular compact set
(regular with respect to the Dirichlet problem in $\overline{\mathbf{C}}\setminus K$), then
the Bernstein-Walsh lemma
(see e.g. \cite[p.\ 77]{Walsh} or \cite[Theorem 5.5.7]{Ransford})
and Cauchy's formula for the derivative
of an analytic function easily give that
\[\left\|{P_n'}\right\|_{K}=e^{o(n)}\|P_n\|_K.\]
This implies (\ref{global}) for the polynomial $P_n'(x)/n\|P_n\|_K$.
On the other hand, \cite[Theorem 3.1]{Totikacta} (see also
\cite{Baran}) claims that on the interior of $K$
we have
\begin{equation} 
|P_n'(x)|\le n\pi\omega_K(x)\|P_n\|_K,\label{Bernstein}
\end{equation} 
hence
\begin{equation} 
\left|\frac{P_n'(x)}{n\|P_n\|_K}\right|\le \frac{h(x)}{\sqrt{a-x}},\qquad x\in [a-\rho,a],
\label{fd}
\end{equation} 
where $h(x)=\sqrt {a-x}\pi \omega_K(x)$ on $[a-\rho,a]$ (and extend this $h$ to
an arbitrary continuous and positive function from there). This
is condition (\ref{local}) for the polynomial $P_n'(x)/n\|P_n\|_K$.
Furthermore, here $h(a)=\pi\Omega(K,a)$.
Therefore, we can apply Theorem \ref{thschur} to conclude
that
\[ 
\left\|\frac{P_n'}{n\|P_n\|_K}\right\|_{[a-\rho,a]}\le n(1+o(1))2\pi\Bigl(\pi\Omega(K,a)\Bigr) \Omega(K,a),
\]
which implies
\begin{equation}  
\|P_n'\|_{[a-\rho,a]}\le n^2(1+o(1))2\pi^2\Omega(K,a)^2\|P_n\|_K,\label{hhh}
\end{equation} 
which is precisely the inequality $M(K,a)\le 2\pi^2\Omega(K,a)^2$ in Theorem \ref{thmain}.
When $K$ is not regular, in the
reasoning above,
instead of $K$, just use the sets $K_m^-\subset K$ from
(\ref{knm}) to be introduced in Section \ref{Sec2}, make the conclusion
\begin{multline*}
\|P_n'\|_{[a-\rho,a]}\le n^2(1+o(1))2\pi^2\Omega(K_m^-,a)^2\|P_n\|_{K_m^-}
\\
\le n^2(1+o(1))2\pi^2\Omega(K_m^-,a)^2\|P_n\|_{K},
\end{multline*}
instead of (\ref{hhh}), and use the fact that, by Proposition
\ref{prop:convergence} below, the quantity
$ \Omega(K_m^-,a)$ is as close to $\Omega(K,a)$ as we wish if $m$ is sufficiently large.

\medskip

The quantity $\Omega(K,a)$ has been formulated in terms of the equilibrium density, but we can
give a direct formulation as follows. $\mathbf{R}\setminus K$ is the union of countably many
open intervals: $\mathbf{R}\setminus K=\cup_{j=0}^\infty I_j$, where, say, $I_0$ and $I_1$ are
the two unbounded complementary intervals (if $K$ itself consists of finitely many intervals,
then the preceding union should be replaced by finite one). We may also assume that
the numbering is such that $I_2$ contains
$(a,a+2\rho)$ (if $I_0\cup I_1$ does not do so). For $m\ge 2$ consider the set
\begin{equation} K^+_m=\mathbf{R}\setminus \left(\bigcup_{j=0}^m I_j\right).\label{knp}\end{equation} 
This contains $K$, it satisfies the interval condition (\ref{intc}), and it
consists of $m$ disjoint closed intervals: $K_m^+=\cup_{j=1}^m[a_{j,m},b_{j,m}]$ with
$a_{1,m}\le b_{2,m}<a_{2,m}\le b_{2,m}<\cdots<a_{m,m}\le b_{m,m}$.
When $a_{j,m}=b_{j,m}$ for some $j$, then the corresponding interval
is degenerated, and we can drop it from the consideration below, so
we may assume $a_{j,m}<b_{j,m}$ for all $j=1,\ldots,m$.
The equilibrium density of $K_m^+$ is (see e.g. \cite[Lemma 2.3]{Totikacta})
\begin{equation} 
\omega_{K_m^+}(x)=\frac{\prod_{j=1}^{m-1}|x-
\lambda_{j,m}|}{\pi\sqrt{\prod_{j=1}^m|x-a_{j,m}||x-b_{j,m}|}},\qquad x\in {\rm Int}(K_m^+),
\label{omega}
\end{equation} 
where $\lambda_{j,m}$ are chosen so that
\begin{equation} \int_{b_{k,m}}^{a_{k+1,m}}\frac{\prod_{j=1}^{m-1}(t-
\lambda_{j,m})}{\sqrt{\prod_{j=1}^m|t-a_{j,m}||t-b_{j,m}|}}dt=0
\label{omega1}\end{equation} 
for all $k=1,\ldots,m-1$. It can be easily shown that these
$\lambda_{j,m}$'s  are
uniquely determined and there is one $\lambda_{j,m}$ on every contiguous
interval
$(b_{k,m},a_{k+1,m})$.
Now $a$ is one of the $b_{j,m}$'s, say $a=b_{j_0,m}$, and then
clearly
\begin{equation} 
\Omega(K_m^+,a)=\frac{\prod_{j=1}^{m-1}|a-
\lambda_{j,m}|}{\pi\sqrt{\prod_{j=1}^m|a-a_{j,m}|}\sqrt{\prod_{j=1,\ j\not=j_0}^m|a-b_{j,m}|}}.
\label{finite}
\end{equation} 

When $K$ consists of finitely many intervals, i.e. $K=K_m^+$ for some $m$,
then this gives an explicit expression for $\Omega(K,a)$.
In the general case, since $K_{m+1}^+\subset K_m^+$, the equilibrium measure $\nu_{K_{m+1}^+}$ is the
balayage of  $\nu_{K_{m}^+}$ onto  $K_{m+1}^+$ (see \cite[Theorem IV.1.6,e]{SaffTotik}),
hence
$\omega_{K_{m+1}^+}(t)\ge \omega_{K_{m}^+}(t)$ for all $t\in {\rm Int}(K_{m+1}^+)$.
As a consequence, the sequence $\{\Omega(K_m^+,a)\}_{m=2}^\infty$ is increasing,
and we shall verify in the next section that
\[\Omega(K,a)=\lim_{m\rightarrow\infty}\Omega(K_m^+,a).\]

The just used monotonicity argument will be used later:
if $K\subset S$ both satisfy the interval condition
(\ref{intc}), then
\begin{equation} \Omega(S,a)\le \Omega(K,a).\label{mon}\end{equation}

\section{Properties of $\Omega(K,a)$}\label{Sec2}

First, we show that the limit $\Omega\left(K,a\right)$ exists in
a uniform way.

Let
\[
\mathcal{E}:=\Bigl\{K\subset\bR\sep K\mbox{ compact, satisfies (\ref{intc})}\Bigr\}.
\]
\begin{Lem}\label{lem:uniformity}
For all $K\in\mathcal{E}$ there exists $L_{K}\in\left(0,\infty\right)$
such that
\[
\lim_{t\rightarrow a-0}\omega_{K}\left(t\right)\left|t-a\right|^{1/2}=L_{K}.
\]
Moreover, this convergence is uniform in $K\in\mathcal{E}$: for every
$\varepsilon>0$ there exists $\delta>0$ such that for all $K\in\mathcal{E}$,
$t\in\left(a-\delta,a\right)$ we
have
\[
\left|\omega_{K}\left(t\right)\left|t-a\right|^{1/2}-L_{K}\right|<\varepsilon.
\]
\end{Lem}
\begin{proof}
Let $\delta_x$ denote the Dirac measure at $x$ and let
$\Bal\left(\delta_x,\left[b,a\right];t\right)$
denote the density at $t$ of the balayage of $\delta_x$
onto $\left[b,a\right]$, $b<a$.
Sometimes we also use the same notation for the measure: $\Bal\left(\delta_x,\left[b,a\right];H\right)$
denotes the balayage measure of the Borel-set $H$. We use \cite[(4.47), Ch.II]{SaffTotik}:
\begin{equation}
\Bal\left(\delta_x,\left[b,a\right];t\right)=\frac{1}{\pi}\frac{\sqrt{\left|x-b\right|
\left|x-a\right|}}{\left|t-x\right|\sqrt{\left|t-a\right|\left|t-b\right|}}.\label{eq:diracbalayage}
\end{equation}
Thus, in this case clearly
\begin{equation} 
\lim_{t\rightarrow a-0}\Bal\left(\delta_x,\left[b,a\right];t\right)|t-a|^{1/2}
=\frac{1}{\pi}\frac{\sqrt{\left|x-b\right|}}{\sqrt{\left|x-a\right|\left|a-b\right|}}=:L_x.
\label{lx}\end{equation} 
Below we set $b=a-\rho$.
The family
\[\Bigl\{ \Bal\left(\delta_x,[a-\rho,a];t\right)\sep x\in\bR\setminus\left[a-2\rho,a+\rho\right]\Bigr\} \]
is uniform in the sense that for every $\varepsilon>0$ there is a $\delta>0$ such that for all $x\in\bR\setminus\left[a-2\rho,a+\rho\right]$
and for all $t\in\left(a-\delta,a\right)$
we have
\begin{multline}
\left|\sqrt{|a-t|}\cdot\Bal\left(\delta_x,[a-\rho,a];t\right)-
\frac{1}{\pi}\frac{\sqrt{\left|x-b\right|}}{\sqrt{\left|x-a\right|\left|a-b\right|}}\right|
\\
 =
\left|\frac{1}{\pi}\frac{\sqrt{\left|x-b\right|\left|x-a\right|}}{\left|t-x\right|\sqrt{\left|t-b\right|}}-
\frac{1}{\pi}\frac{\sqrt{\left|x-b\right|}}{\sqrt{\left|x-a\right|\left|a-b\right|}}\right|
<\varepsilon,\qquad b=a-\rho.
\label{uniformity_for_balayage}
\end{multline}
This is a simple calculus exercise, we skip it.

If $\mu$ is any positive Borel-measure with compact
support and $\mathrm{supp}\left(\mu\right)\subset\bR\setminus[a-2\rho,a+\rho]$
and $\mu\left(\bR\right)\le1$, then on $(a-\rho,a)$ the density function of the measure
\[\mu^{*}(\cdot):=\int\Bal\left(\delta_x,\left[a-\rho,a\right];\cdot\right)\, d\mu\left(x\right)\]
has the form
\[\frac{d\mu^{*}\left(t\right)}{dt}=\int\Bal\left(\delta_x,\left[a-\rho,a\right];\, t\right)\, d\mu\left(x\right).\]
This follows from (\ref{eq:diracbalayage}):
if we fix $t\in (a-\rho,a)$ then this is away from the support of $\mu^*$, so
\[\Bal\left(\delta_x,\left[a-\rho,a\right];\left[t,t+u\right]\right)/u,\qquad 0<u<\frac{a-t}{2}\]
is bounded and we can use Lebesgue's dominated convergence theorem when taking the limit for $u\rightarrow 0$.

We can rewrite (\ref{uniformity_for_balayage}) as
\[
\left|\sqrt{ |a-t|}\cdot \Bal\left(\delta_x,\left[a-\rho,a\right];\, t\right)-L_{x}\right|
<\varepsilon,\qquad t\in (a-\delta,a).
\]
 Now if this inequality is
integrated with respect to $\mu$ we obtain that the limit
\[\lim_{t\rightarrow a-0}\sqrt{|t-a|}\cdot\frac{d\mu^{*}\left(t\right)}{dt}\]
exists uniformly in the measure $\mu^*$.

Finally, we use that $\nu_{[a-\rho,a]}$  is the balayage of $\nu_K$ onto $[a-\rho,a]$
(see \cite[Theorem IV.1.6,e]{SaffTotik}). When forming this balayage measure the part of
$\nu_K$ that is on $[a-\rho,a]$ is left unchanged, and the rest of
$\nu_K$ is swept onto $[a-\rho,a]$,
and this latter balayage measure is
\[\Bal({\nu_K\vert_{K\setminus [a-\rho,a]}},[a-\rho,a];H)=\int_{K\setminus [a-\rho,a]}
\Bal\left(\delta_x,\left[a-\rho,a\right];H\right)d\nu_K(x)\]
where $H$ is arbitrary (Borel-) set.
Thus, we have the formula
\[
\nu_{\left[a-\rho,a\right]}\left(H\right)={\nu_{K}\vert_{\left[a-\rho,a\right]}}\left(H\right)+
\int_{K\setminus\left[a-\rho,a\right]}\Bal\left(\delta_x,\left[a-\rho,a\right];H\right)\, d\nu_{K}\left(x\right).
\]
Rewriting this for the densities,
we have for $t\in\left(a-\rho,a\right)$
\[
\omega_{K}\left(t\right) =\omega_{\left[a-\rho,a\right]}\left(t\right) -
\int_{K\setminus\left[a-\rho,a\right]}\Bal\left(\delta_x,\left[a-\rho,a\right];t\right)\,
d\nu_{K}\left(x\right).
\]
Now if $\sigma$ is either of the terms
on the right-hand side, then, as we have just seen, the limit
\[\lim_{t\rightarrow a-0} \sqrt{|t-a|}\cdot\sigma(t)\]
exists, and this limit is uniform in the set $K\in {\mathcal{E}}$.
This proves the claim in the lemma.\end{proof}

We shall need to use a theorem of Ancona \cite{MR890353}:
Let $K\subset\bR$ be a compact set of positive logarithmic capacity.
Then, for every $m$, there exists a regular compact set (regular with respect
to the solution of the Dirichlet problem in its complement relative to $\overline{\mathbf{C}}$)
$K_m^{-}\subset K$ such that
\begin{equation} \mathrm{cap}\left(K\right)\le\mathrm{cap}(K_m^{-})+1/m.\label{knm}\end{equation} 
Since the union of two regular compact sets is regular
and $K$ satisfies the interval condition (\ref{intc}),
we may assume that $K_m^-$ also satisfies that condition
(if not, just unite it with $[a-2\rho,a]$), and also that $K_m^-\subseteq K_{m+1}^-$.

\begin{Lem}
\label{lem:kminus_weakstar} For the sets  $K_m^{\pm}$
from (\ref{knp}) and (\ref{knm}) we have
$\nu_{K_m^{\pm}}\rightarrow\nu_{K}$ in weak$^*$ sense as
$m\rightarrow\infty$. \end{Lem}
\begin{proof} 
From the monotone convergence of $K_m^\pm$ to $K$ it follows
that $\c(K_m^\pm)\rightarrow \c(K)$ as $m\rightarrow\infty$, see \cite[Theorem 5.1.3]{Ransford}.

Let $\nu_{K_m^{\pm}}{\rightarrow}\nu$, $m\rightarrow\infty$, $m\in {\mathcal{N}}$ (for
some subsequence ${\mathcal{N}}$ of the natural numbers), be a weak$^*$-limit of the sequence
$\{\nu_{K_m^{\pm}}\}$. Clearly, $\nu$ is supported
on $K$ and it has total mass 1.
Let
\[I(\mu)=\int\int \log\frac{1}{|z-t|}d\mu(t)d\mu(z)\]
be the logarithmic energy of a measure $\mu$. The equilibrium measure
$\nu_K$ minimizes this energy among all probablity measures on $K$, and
with it we have the formula
$\c(K)=\exp(-I(\nu_K))$ (see \cite[Definition 5.1.1]{Ransford}).
Now it follows
from $\c(K_m^{\pm})\rightarrow \c(K)$ and from the
principle of descent (see e.g. \cite[Theorem I.6.8]{SaffTotik}),
 that
\[
I(\nu_{K})=\lim_{m\rightarrow\infty,\ m\in {\mathcal{N}}}I\left(\nu_{K_m^{\pm}}\right)
=\liminf_{m\rightarrow\infty,\ m\in {\mathcal{N}}}I\left(\nu_{K_m^{\pm}}\right)\geq I\left(\nu\right)\geq I(\nu_{K}).
\]
But the equilibrium measure $\nu_K$ is the unique measure to minimize the logarithmic
energy among all unit Borel-measures with support on $K$, hence $\nu$ must be
equal to $\nu_K$. Since this is true for all weak$^*$-convergent subsequences of
$\{\nu_{K_m^{\pm}}\}$, the claim in the lemma follows.
\end{proof}

Finally, we verify

\begin{Prop}
\label{prop:convergence} For the sets  $K_m^{\pm}$
from (\ref{knp}) and (\ref{knm}) we have that $\Omega\left(K_{m}^\pm,a\right)\rightarrow\Omega\left(K,a\right)$
as $m\rightarrow\infty$.
\end{Prop}

\begin{proof} 
Denote Green's function of $\overline{ \mathbf{C}}\setminus K$ with pole at infinity
by $g_{K}\left(z\right)$.  It has the form
\begin{equation} g_K(z)=\int \log |z-t| d\nu_K\left(t\right)-\log\c(K),\label{form}\end{equation} 
see \cite[Sec. 4.4]{Ransford} or formula \cite[(I.4.8)]{SaffTotik}.
Consider the function
\[u_m(z)=g_{K_m^\pm}(z)-g_K(z).\]
This is harmonic in $\overline{\mathbf{C}}\setminus (K\cup K_m^\pm)$.
In view of Lemma \ref{lem:kminus_weakstar}, of $\c(K_m^\pm)\rightarrow \c(K)$ and of (\ref{form}),
the functions $u_m(z)$ tend to 0
uniformly on compact subsets of $\overline{ \mathbf{C}}\setminus K$  as $m\rightarrow\infty$.
Since $[a-2\rho,a]$ is part of all the sets $K_m^{\pm}$, we have
$g_{K_m^\pm}(z),\ g_K(z)\le g_{[a-2\rho,a]}(z)$, so $u_m(z)\rightarrow 0$ as $z\rightarrow a-\rho$,
and this convergence is uniform in $m$.
Thus, $u_m(z)\rightarrow 0$  uniformly on the circle
 $C_{\rho}(a):=\{z\sep |z-a|=\rho\}$. Let $D_\rho(a)$ be the interior of that
 circle.

In what follows we shall use the main branch of the square root function.
$\zeta=i\sqrt{z-a}$ maps $D_\rho(a)\setminus [a-\rho,a]$ onto the upper half-disk
$\{\zeta\sep |\zeta|<\sqrt{\rho},\ \Im \zeta>0\}$, so $v_m(\zeta)=u_m(a-\zeta^2)$ is
a harmonic function there, which vanishes on $(-\sqrt{\rho},\sqrt{\rho})$.
By the reflection principle we can extend it to a harmonic function
in the disk $D_{\sqrt{\rho}}(0)$. From the fact that $u_m\rightarrow 0$
uniformly on $C_\rho(a)$ it is immediate that $v_m\rightarrow 0$ uniformly
on $C_{\sqrt{\rho}}(0)$, so its partial derivative
in the $i$-direction
tends to 0 uniformly on compact subsets of
 $D_{\sqrt{\rho}}(0)$. For $x\in (a-\rho,a)$ we have
\[
\frac{\partial}{\partial\mathbf{n}_+}v_{m}\left(w\right)
=\frac{\partial}{\partial\mathbf{n}_{+}}u_{m}\left(x\right)\cdot2\sqrt{a-x},\qquad
w=-\sqrt{a-x},
\]
where $\partial u_m(x)/\partial {\bf n}^+$ denotes the derivative of $u_m$ with
 respect to the upper normal to $\mathbf{R}$ at $x$.
 As we have just mentioned, the left-hand side
 tends to 0 uniformly on compact subsets of
 $D_{\sqrt{\rho}}(0)\cap \mathbf{R}$.
 It follows that
 $\sqrt{a-x} \cdot (\partial u_m(x)/\partial {\bf n}^+)$ tends uniformly to 0 on $[a-\rho/2,a]$.
 Since
\[
\frac{\partial}{\partial\mathbf{n}_{+}}u_{m}\left(z\right)=
\frac{\partial}{\partial\mathbf{n}_{+}}g_{K_{m}^\pm}\left(z\right)-
\frac{\partial}{\partial\mathbf{n}_{+}}g_{K}\left(z\right)=\pi\omega_{K_{m}^\pm}\left(z\right)-\pi\omega_{K}\left(z\right)
\]
(see e.g. \cite[II.(4.1)]{Nevanlinna}),
this proves the Proposition. \end{proof}

\section{Proof of Theorem \ref{thmain}}
When $K$ consists of finitely many intervals like the sets $K_m^+$ in
(\ref{knp}), the theorem
follows from  \cite[Theorem 4.1]{Totikacta} and from (\ref{finite}).

\begin{proof}[Proof of $M(K,a)\le 2\pi^2\Omega(K,a)^2$]
First we prove this inequality when $K$ is regular with respect to the Dirichlet
problem in $\overline{ \mathbf{C}}\setminus K$, and later we remove
the regularity condition.

So assume that  $K$ is regular and satisfies  the interval condition
(\ref{intc}).  Fix $\varepsilon>0$, and let
$K^{+}:=[\min K,\max K]$.
There exist (see e.g. \cite[Corollary VI.3.6]{SaffTotik})  $0<\tau<1$
and polynomial $Q_{n\varepsilon}$ of $\deg\left(Q_{n\varepsilon}\right)\le n\varepsilon$
such that
\begin{enumerate}[a)]
\item $1-e^{-n\tau}\le Q_{n\varepsilon}\left(x\right)\le1$ if $x\in [a-\rho,a]$,
\item $0\le Q_{n\varepsilon}\left(x\right)\le1$ if $x\in [a-3\rho/2,a-\rho)\cup(a,a+3\rho/2]$,
\item $|Q_{n\varepsilon}\left(x\right)|\le e^{-n\tau}$
if $x\in K^{+}\setminus [a-3\rho/2,a+3\rho/2]$.
\end{enumerate}
In particular,  $\left\Vert Q_{n\varepsilon}\right\Vert _{K^{+}}\le1$.
Let $g_K$ denote Green's function of $\overline{\mathbf{C}}\setminus K$ with pole at infinity.
The
regularity of $K$ implies that $g_{K}$ is continuous and vanishes
on $K$. Hence, there
exists $0<\theta<1$, $\theta=\theta\left(\tau\right)$, such that
\begin{equation}
\mbox{if }x\in\bR,\ \mathrm{dist}\!\left(x,K\right)\le\theta,\mbox{ then }g_{K}\left(x\right)\le\tau^{2}.\label{green_on_kstar}
\end{equation}

Choose a large $m$ such that for the sets $K_m^+$ from
(\ref{knp}) we have {\rm dist}$(x,K)<\theta$ for all $x\in K_m^+$.
We are going to apply \cite[Theorem 4.1]{Totikacta} for the polynomial
$P_{n}Q_{n\varepsilon}$ on $K_m^{+}$ where $P_{n}$ is an arbitrary polynomial
with degree $n$. Then $P_{n}Q_{n\varepsilon}$ is a polynomial of
degree at most $\left(1+\varepsilon\right)n$ and we estimate its
sup-norm on $K_m^{+}$ as follows.
First,
if $x\in K$, then $\left|P_{n}\left(x\right)Q_{n\varepsilon}\left(x\right)\right|\le\left\Vert P_{n}\right\Vert _{K}$,
see properties a)--c) above.
Second, if $x\in K_m^{+}\setminus K$,
then we apply the Bernstein-Walsh lemma
(see e.g. \cite[p.\ 77]{Walsh} or \cite[Theorem 5.5.7]{Ransford}) for $P_{n}$ and property c)
for $Q_{n\varepsilon}$, as well as (\ref{green_on_kstar}) to obtain
\begin{multline*}
\left|P_{n}\left(x\right)Q_{n\varepsilon}\left(x\right)\right|
\le
\left\Vert P_{n}\right\Vert _{K}\exp\left(n\, g_{K}\left(x\right)\right)\exp\left(-n\tau\right)
\\
\le \left\Vert P_{n}\right\Vert _{K}\exp\left(n\tau^{2}\right)\exp\left(-n\tau\right)
\le\left\Vert P_{n}\right\Vert _{K}.
\end{multline*}
Hence,
\begin{equation}
\|P_{n}Q_{n\varepsilon}\|_{K_m^+}\le\left\Vert P_{n}\right\Vert _{K}.\label{eq:pq_estimate}
\end{equation}

Next, for $x\in [a-\rho,a]$
\begin{multline*}
\left|\left(P_{n}Q_{n\varepsilon}\right)'\left(x\right)\right|
\ge
\left|P_{n}'\left(x\right)Q_{\varepsilon n}(x)\right|-\left|P_{n}\left(x\right)Q_{n\varepsilon}'\left(x\right)\right|
\\
\ge\left|P_{n}'\left(x\right)\right|\left(1-e^{-n\tau}\right)-\left|P_{n}\left(x\right)Q_{n\varepsilon}'\left(x\right)\right|
\end{multline*}
 and here we can use the (transformed form of the) Markov inequality (\ref{eq:classicmarkov}) to conclude
 \[\left\Vert Q_{n\varepsilon}'\right\Vert_{K^+} \le C_{1}\varepsilon^{2}n^{2}\]
with some constant $C_1$.
Therefore, for $x\in [a-\rho,a]$
\begin{multline*}
\left|P_{n}'\left(x\right)\right|\left(1-e^{-n\tau}\right)
\le
\left|
\left(P_{n}Q_{n\varepsilon}\right)'\left(x\right)\right|+\left|P_{n}\left(x\right)Q_{n\varepsilon}'
\left(x\right)\right|
\\
\le
\left|\left(P_{n}Q_{n\varepsilon}\right)'\left(x\right)\right|+\left\Vert P_{n}
\right\Vert _{K}C_{1}\varepsilon^{2}n^{2}.
\end{multline*}
Now we use that, as has already been mentioned,
the theorem is true for the set $K_m^+$ since it consists of finitely
many intervals.
Hence, we can continue the preceding estimate as
\begin{multline*}
\le
(\left(1+\varepsilon\right)n)^{2}\left(1+o_{K_m^{+}}\left(1\right)\right)2\pi^2\Omega\left(K_m^{+},a\right)^2
\left\Vert P_{n}Q_{n\varepsilon}\right\Vert _{K_m^{+}}
+\left\Vert P_{n}\right\Vert _{K}C_{1}\varepsilon^{2}n^{2}
\\
\le n^{2}\left\Vert P_{n}\right\Vert _{K}\left(\left(1+o_{K_m^{+}}\left(1\right)\right)\left(1+\varepsilon\right)^{2}2\pi^2\Omega\left(K_m^{+},a\right)^2+C_{1}\varepsilon^{2}\right),
\end{multline*}
where we also used (\ref{eq:pq_estimate}).
On applying the monotonicity (\ref{mon}) of $\Omega\left(.,a\right)$ we can continue
the preceding estimates as
\[
\le n^{2}\left\Vert P_{n}\right\Vert _{K}\left(\left(1+o_{K}\left(1\right)\right)\left(1+\varepsilon\right)^{2}2\pi^2\Omega\left(K,a\right)^2+C_{1}\varepsilon^{2}\right).
\]
Since here $\varepsilon>0$ is arbitrary, the inequality $M(K,a)\le 2\pi^2\Omega(K,a)^2$ follows for regular $K$
from the
just given chain of inequalities.

To remove the regularity condition consider the sets
$K_m^-$ from (\ref{knm}). These are regular sets satisfying the interval condition
(\ref{intc}), so we can apply the just proven estimate to them:
\begin{multline*}
\left\Vert P_{n}'\right\Vert _{[a-\rho,a]}
\le
 n^{2}\left\Vert P_{n}\right\Vert _{K_m^{-}}\left(1+o_{K_m^{-}}\left(1\right)\right)2\pi^2\Omega\left(K_m^{-},a\right)^2
\\
\leq n^{2}\left\Vert P_{n}\right\Vert _{K}\left(1+o_{K_m^{-}}\left(1\right)\right)2\pi^2\Omega\left(K_m^{-},a\right)^2,
\end{multline*}
where we used $K_m^{-}\subset K$, and
hence $\left\Vert P_{n}\right\Vert _{K_m^{-}}\le\left\Vert P_{n}\right\Vert _{K}$.
Since on the right $\Omega\left(K_m^{-},a\right)$ can
be made arbitrarily close to $\Omega\left(K,a\right)$
by choosing a large $m$ (see Proposition \ref{prop:convergence}),
the inequality $M(K,a)\le 2\pi^2\Omega(K,a)^2$ follows in the
general case.\end{proof}

\begin{proof}[Proof of $M(K,a)\ge 2\pi^2\Omega(K,a)^2$]
We construct a sequence of polynomials $\left\{ P_{n}\right\} _{n=1}^\infty$, ${\rm deg}(P_n)=n$,
such that
\begin{equation} 
\frac{\left|P_{n}'\left(a\right)\right|}{n^2\left\Vert P_{n}\right\Vert _{K}}
\rightarrow2\pi^2\Omega\left(K,a\right)^2\quad \mbox{ as }n\rightarrow\infty.
\label{oi}\end{equation} 
Consider $K_m^{+}$ from (\ref{knp}) for some integer $m$. It is the  union of finitely many intervals,
but some of them may be degenerated, i.e. some of them may be a singleton.
Replace each such singletons in $K_m^+$ by an interval of length $<1/m$
(alternatively, for $m>1/\rho$  we could set $K_m^+$ as the
set $\{x\sep {\rm dist}(x,K)\le 1/m\}\setminus (a,a+2\rho)$).
The resulting set, which we continue to denote by $K_m^+$, consists of non-degenerated intervals,
so we can apply the sharpness result
in \cite{Totikacta}, formula (4.7) on page 155, according to which there is a
 sequence $\left\{ P_{m,n}\right\} _{n=1}^\infty$, ${\rm deg}(P_{m,n})=n$, of polynomials  such that
\[
\left|P_{m,n}'\left(a\right)\right|\ge\left(1-o_{K_m^+}(1)\right)
2\pi^2\Omega\left(K_m^{+},a\right)^2 n^2\left\Vert P_{m,n}\right\Vert _{K_m^{+}},
\]
where $o_{K_m^+}(1)$ depends on ${K_m^+}$ and it tends to $0$ as $n\rightarrow\infty$
for each fixed $m$. Since $K\subset K_m^{+}$, we have
$\left\Vert P_{m,n}\right\Vert _{K_m^{+}}\ge\left\Vert P_{m,n}\right\Vert _{K}$,
so
\[
\left|P_{m,n}'\left(a\right)\right|\ge\left(1-o_{K_m^+}(1)\right)2\pi^2\Omega\left(K_m^{+},a\right)^2n^{2}
\left\Vert P_{m,n}\right\Vert _{K}.
\]
Since here  $\Omega\left(K_m^{+},a\right)$ can be made arbitrarily close
to $\Omega\left(K,a\right)$ by selecting a sufficiently large $m$  (see Proposition \ref{prop:convergence},
which holds true also for these modified sets $K_m^+$),
the relation (\ref{oi})
follows for $P_n=P_{m_n,n}$ if $m_n$ tends to infinity  sufficiently slowly as $n\rightarrow\infty$.
\end{proof}

\section{Proof of Theorem \ref{thschur}}
First we need to verify Theorem \ref{thschur} in the special case when $K$ consists of finitely many intervals.
In this case we use the polynomial inverse image technique of \cite{TotikSA}, and
deduce the theorem from Schur's inequality (\ref{Schur}).

\begin{proof}[Proof of  Theorem \ref{thschur} when $K$ consists of finitely many intervals]
  First   we   deal with the estimate (\ref{conc}).

Let $K=\cup_{j=1}^l[a_{2j-1},a_{2j}]$. For any $\varepsilon>0$ there is a set
$K^*=\cup_1^l[a_{2j-1}^*,a_{2j}]$ such that
\begin{equation} a_{2j-1}-\varepsilon<a_{2j-1}^*<a_{2j-1}\qquad \mbox{ for all
$j$},\label{dd}\end{equation} 
 and $K^*$ is the complete inverse image of $[-1,1]$ under a polynomial
$T_N$ of some degree $N$: $K^*=T_N^{-1}[-1,1]$, see \cite[Theorem 2.1]{Totikacta}
(cf. also the history of this density theorem in \cite{TotikSA}). This $T_N$ then has $N$ zeros on $K^*$,
and $T_N(x)$ runs through the interval $[-1,1]$ precisely $N$ times as $x$ runs through $K^*$.
Thus, there are intervals $E_1,\ldots,E_N\subset K^*$, $K^*=\cup_{k=1}^N E_k$,
 that are disjoint except perhaps
for their endpoints and  $T_N$ is a bijection from each $E_k$ onto $[-1,1]$.
The point $a$ is the right-endpoint of one of these intervals, say of
$E_1$ (the numbering of the $E_k$'s is arbitrary). 
The equilibrium density of $K^*$ has the form (see \cite[(3.8)]{Totikacta})
\[
\omega_{K^*}(t)=\frac{|T_N'(t)|}{N\pi\sqrt{1-T_N^2(t)}},
\]
which easily implies that
\begin{equation} \Omega(K^*,a)=\frac{|T_N'(a)|^{1/2}}{\sqrt 2 \pi N}.\label{tnp}\end{equation} 

 For an $\eta>0$ choose $\delta>0$ so that for all $t\in [a-\delta,a]$
\begin{equation} \frac{\sqrt{2 |T_N'(a)|}}{(1+\eta)\sqrt{1-T_N(t)^2}}\le \frac{1}{\sqrt{a-t}}\le
(1+\eta)\frac{\sqrt{2 |T_N'(a)|}}{\sqrt{1-T_N(t)^2}}\label{po}\end{equation} 
(this is possible since $T_N^2(a)=1$ and $T_N'(a)\not=0$), and
\begin{equation} \frac{h(a)}{1+\eta}\le h(t)\le (1+\eta)h(a)\label{po1}\end{equation} 
are true. We may also suppose that $\delta$ is smaller than
$\rho$ (see (\ref{intc})) and smaller than the quarter-length of $E_1$.
For an $n$ choose (see e.g. \cite[Corollary VI.3.6]{SaffTotik}) polynomials
 $Q_{\varepsilon n}$ of degree at most $\varepsilon n$ such that with some $0<q<1$ we have
\begin{enumerate}[(i)]
\item $|Q_{\varepsilon n}(t)-1|<q^n$ if $t\in [a-\delta,a]$,
\item $0\le Q_{\varepsilon n}(t)\le 1$ on $[a-2\delta,a-\delta]$, and
\item $0\le Q_{\varepsilon n}(t)<q^n$ if $t\in K^*\setminus [a-2\delta,a]$.
\end{enumerate}

For $t\in E_1$ let $t_k=t_k(t)\in E_k$ be the point with
$T_N(t)=T_N(t_k)$. Now if $P_n$ is a polynomial as in the theorem, then we set for $t\in E_1$
\[
S_n(t)=\sum_{k=1}^N (P_nQ_{\varepsilon n})(t_k).
\]
Actually, $S_n(t)$ is a polynomial with degree at most $n+\varepsilon n$, see \cite{Totikacta} formula (3.13).
Note that all $t_k$, $k\ge 2$ are outside the interval $[a-2\delta,a]$,
hence, in view of (\ref{global}) and (iii),  with any $0<q<q_1<1$ we have
the relation
\begin{equation} 
S_n(x)=P_n(x)+O(q_1^n), \qquad x\in E_1,\label{fg}\end{equation} 
furthermore for $x\in E_1\setminus [a-2\delta,a]$ we even have $S_n(x)=O(q_1^n)$.
Thus, in view of the assumption (\ref{local}), for $x\in [a-\delta,a]$
we get from (\ref{po}) and (\ref{po1}) that
\[|S_n(x)|\le \frac{h(x)}{\sqrt {a-x}}+O(q_1^n)
\le (1+\eta)^2 h(a) \frac{\sqrt{2|T_N'(a)|}}{\sqrt{1-T_N^2(x)}}+O(q_1^n),\]
which gives
\begin{equation} |S_n(x)|\le  (1+\eta)^3 h(a) \frac{\sqrt{2|T_N'(a)|}}{\sqrt{1-T_N^2(x)}}\label{ui}\end{equation} 
for all large $n$.

In a similar manner, if $x\in K\setminus[a-2\delta,a]$, then
\[|S_n(x)|=O(q_1^n)\le  (1+\eta)^3 h(a) \frac{\sqrt{2|T_N'(a)|}}{\sqrt{1-T_N^2(x)}}\]
for all large $n$, i.e. (\ref{ui}) is true for all $x\in E_1$.

Now $S_n(x)=V_m(T_N(x))$ with some polynomial $V_m$ of degree at most
$\deg (P_n Q_{\varepsilon n})/N\le (1+\varepsilon)n/N$ (see \cite[Section 5]{TotikSA}), and then
(\ref{ui}) can be written in the form
\[V_m(w)\le (1+\eta)^3 h(a) \frac{\sqrt{2|T_N'(a)|}}{\sqrt{1-w^2}}, \qquad w\in (-1,1).\]
Upon applying the Schur inequality (\ref{Schur0})--(\ref{Schur}) we obtain from $(m+1)\le (1+2\varepsilon)n/N$
(which certainly holds for large $n$)
\[\|V_m\|_{[-1,1]}\le (1+\eta)^3 h(a) \sqrt{2|T_N'(a)|}(m+1)\le
(1+\eta)^3 (1+2\varepsilon) h(a) \sqrt{2|T_N'(a)|}n/N.\]
Using (\ref{tnp}), (\ref{fg}) and $S_n(x)=V_m(T_N(x))$, we can conclude that
\[|P_n(x)+O(q_1^n)|\le (1+\eta)^3 (1+2\varepsilon) h(a) 2\pi\Omega(K^*,a)n,\qquad x\in E_1,\]
which gives
\begin{equation} |P_n(x)|\le (1+\eta)^4 (1+2\varepsilon) h(a) 2\pi\Omega(K^*,a)n,\qquad x\in E_1,\label{ty}\end{equation} 
for sufficiently large $n$. Finally, using the monotonicity property
(\ref{mon}) of $\Omega$ it follows from $K\subset K^*$ that
\begin{equation} |P_n(x)|\le (1+\eta)^4 (1+2\varepsilon) h(a) 2\pi\Omega(K,a)n,\qquad x\in E_1.\label{ty1}\end{equation} 

This is the desired estimate on $E_1$. On $[a-\rho,a]\setminus E_1$ the polynomials $P_n$
are bounded by the assumption (\ref{local}), hence (\ref{ty1}) is true on
all $[a-\rho,a]$ if $n$ is large.  Since $\varepsilon,\eta>0$
in (\ref{ty1}) are also arbitrarily small, the inequality (\ref{conc}) follows.

\bigskip

We still need to prove (\ref{conc1}) in the case considered, i.e. when
 $K=\cup_1^l[a_{2j-1},a_{2j}]$. We use the notations from the preceding proof.

The estimate of Schur in (\ref{Schur}) is sharp: if
${\mathcal{T}}_m(x)=\cos(m\arccos x)$ are the classical Chebyshev polynomials
and $H_m(x)={\mathcal{T}}_{m+1}'(x)/(m+1)$, then (\ref{Schur0}) is true  (use Bernstein's inequality
(\ref{B})) and
\[|H_m(\pm1)|=m+1.\]
Set now
\[P_n(x)=h(a)H_m(T_N(x))U_{\sqrt{ n}}(x)\frac{\sqrt{2 |T_N'(a)|}}{(1+\eta)^2}\]
 where $m=[(n-\sqrt{ n})/N]$
(the integral part of $(n-\sqrt n)/N$) and
$U_{\sqrt {n}}(x)$ is a polynomial of degree smaller than $\sqrt{ n}$
for which $U_{\sqrt{ n}}(a)=1$ and $U_{\sqrt{ n}}(x)\rightarrow 0$ uniformly on
compact subsets of $K^*\setminus \{a\}$.
This is a polynomial of degree at most $n$, and for it we have
\[|P_n(x)|\le h(a)\frac{1}{\sqrt{1-T_N^2(x)}}U_{\sqrt n}(x)\frac{\sqrt{2 |T_N'(a)|}}{(1+\eta)^2}.\]
Using (\ref{po})--(\ref{po1}) and the properties of $U_{\sqrt n}$
it follows that for large $n$ we have
\[|P_n(x)|\le \frac{h(x)}{\sqrt{|a-x|}},\qquad x\in K^*.\]
At the same time, for large $n$,
\begin{multline*} 
|P_n(a)|=
h(a)|H_m(\pm 1)|\frac{\sqrt{2 |T_N'(a)|}}{(1+\eta)^2}
=h(a)(m+1)\frac{\sqrt{2 |T_N'(a)|}}{(1+\eta)^2}
\\
\ge h(a)\frac{n}{(1+\eta) N}\frac{\sqrt{2 |T_N'(a)|}}{(1+\eta)^2},
\end{multline*}
which gives, in view of (\ref{tnp}), the inequality
\[|P_n(a)|\ge  h(a)\frac{n}{(1+\eta)^3}2\pi\Omega(K^*,a).\]
This estimate contains $\Omega(K^*,a)$, and here $K^*$ is a set close to
$K$, but it depends on $\varepsilon>0$ (see (\ref{dd})). With the same argument
that was used in Proposition \ref{prop:convergence} we obtain that on the right-hand side
$\Omega(K^*,a)$ is as close
to $\Omega(K,a)$ as we wish if $\varepsilon>0$ is sufficiently small. Since $\eta>0$ is also
arbitrary, (\ref{conc1}) follows.\end{proof}

\begin{proof}[Proof of Theorem \ref{thschur} for regular sets]

Let now $K\subset \mathbf{R}$ be a regular compact set with the interval condition
(\ref{intc}), and consider the sets $K_m^+$ from
(\ref{knp}). Let $P_n$ be as in the theorem, and with the $Q_{\varepsilon n}$
satisfying properties a)--c) in the proof of Theorem \ref{thmain}
apply the finite interval case of Theorem \ref{thschur}
to the set $K_m^+$ and to the polynomial $P_nQ_{\varepsilon n}$ of
degree at most $(1+\varepsilon)n$. Since $\|P_nQ_{\varepsilon n}\|_{K_m^+}\le \|P_n\|_K$
(see (\ref{eq:pq_estimate})) and $|P_n(x)Q_{\varepsilon n}(x)|\le |P_n(x)|$
for $x\in [a-\rho,a]$, we can conclude that
\begin{multline*}
 \|P_nQ_{\varepsilon n}\|_{[a-\rho,a]}
\le
 n(1+\varepsilon)(1+o(1))2\pi h(a)\Omega(K_m^+,a)
\\
\le
n(1+\varepsilon)(1+o(1))2\pi h(a)\Omega(K,a),
\end{multline*}
where we have also used the monotonicity property (\ref{mon}).
Since $Q_{\varepsilon n}(x)=1-o(1)$ on $[a-\rho,a]$ and $\varepsilon>0$ is arbitrary,
we can conclude (\ref{conc}).

As for (\ref{conc1}), we can choose a sequence $\{P_{m,n}\}_{n=1}^\infty$ for
the set $K_m^+$ as in (\ref{conc1}), i.e. the polynomials
$P_{m,n}$ satisfy (\ref{local})--(\ref{global})
with $P_n$ replaced by $P_{m,n}$, and
\begin{equation} 
|P_{m,n}(a)|\ge n(1-o(1))2\pi h(a)\Omega(K_m^+,a),\qquad n=1,2,\ldots\label{conc1*}
\end{equation} 
By Proposition \ref{prop:convergence} on the right-hand side
the factor $\Omega(K_m^+,a)$ can be
as close as we wish to $\Omega(K,a)$. Hence, if $m=m_n$ tends to infinity
with $n$ but sufficiently slowly, then the polynomials
$P_n=P_{m_n,n}$ satisfy (\ref{local})--(\ref{global}) and (\ref{conc1}).

\end{proof}

\begin{proof}[Proof of Theorem \ref{thschur} for arbitrary sets]
Let now $K\subset \mathbf{R}$ be an arbitrary compact set with the interval condition
(\ref{intc}), and consider the sets $K_m^-$ from
(\ref{knm}). These are regular sets satisfying the same interval condition,
and if $P_n$ are as in the theorem then clearly $P_n$ satisfy the
same conditions on $K_m^-$ instead of $K$. Thus, according to what we
have just proven,
\begin{equation}
  \|P_n\|_{[a-\rho,a]}\le n(1+o(1))2\pi h(a)\Omega(K_m^-,a).\label{c2}
\end{equation} 
On the right-hand side $\Omega(K_m^-,a)$ converges to $\Omega(K,a)$
as $m\rightarrow\infty$ (see Proposition \ref{prop:convergence}), hence (\ref{conc}) can be
concluded from (\ref{c2}).

As for (\ref{conc1}), just repeat the argument given in the preceding proof (with
the modification of $K_m^+$ as at the end of the proof of Theorem \ref{thmain}
when $K_m^+$ contains singletons).\end{proof}

\section*{Acknowledgements}

The first author was supported by the European Research Council Advanced
grant No. 267055, while he had a postdoctoral position at the Bolyai Institute,
University of Szeged, Aradi v. tere 1, Szeged 6720, Hungary.

The second author was supported by Magyary scholarship: This research was
realized in the frames of T\'AMOP 4.2.4. A/2-11-1-2012-0001 ,,National
Excellence Program - Elaborating and operating an inland student and
researcher personal support system.'' The project was subsidized by
the European Union and co-financed by the European Social Fund.

The third author was supported by NSF grant DMS-1265375.

The final publication is available at Springer via \href{http://dx.doi.org/10.1007/s11785-014-0405-z}{http://dx.doi.org/10.1007/s11785-014-0405-z}.

\bigskip{}

\providecommand{\bysame}{\leavevmode\hbox to3em{\hrulefill}\thinspace}
\providecommand{\MR}{\relax\ifhmode\unskip\space\fi MR }
\providecommand{\MRhref}[2]{%
  \href{http://www.ams.org/mathscinet-getitem?mr=#1}{#2}
}
\providecommand{\href}[2]{#2}

\bigskip

Sergei Kalmykov \\
Institute of Applied Mathematics, FEBRAS, 7 Radio Street, Vladivostok,
690041, Russian Federation,  \\
 Far Eastern Federal University, 8 Sukhanova Street, Vladivostok, 690950, Russia and\\
Bolyai Institute, University of Szeged,  Szeged, Aradi v. tere
1, 6720, Hungary \\
email address: \texttt{sergeykalmykov@inbox.ru}
\medskip

B\'ela Nagy\\
MTA-SZTE Analysis and Stochastics Research Group, Bolyai Institute, University of Szeged,
Szeged, Aradi v. tere 1, 6720, Hungary \\
email address: \texttt{nbela@math.u-szeged.hu}
\medskip

Vilmos Totik\\
University of Szeged, Bolyai Institute,
Szeged, Aradi v. tere 1, 6720, Hungary\\
and\\
Department of Mathematics and Statistics, University of South Florida, 4202 E. Fowler Ave.
CMC342, Tampa, FL 33620\\
email address: \texttt{totik@mail.usf.edu}

\end{document}